\documentclass[reqno, english]{amsart}
\usepackage{float}
\usepackage[utf8]{inputenc}
\usepackage{comment}
\usepackage{capt-of}
\usepackage{fullpage}
\usepackage[usenames,dvipsnames,svgnames,table]{xcolor}
\usepackage[colorlinks=true, pdfstartview=FitV, linkcolor=blue, citecolor=blue, urlcolor=darkblue]{hyperref}

\usepackage[maxbibnames=15,doi=false,isbn=false,url=false]{biblatex}
\AtEveryBibitem{
	\clearlist{language}
}

\usepackage[margin=1in]{geometry}                
\usepackage{graphicx}
\usepackage{amssymb}
\usepackage{amsmath}
\usepackage{amsfonts}
\usepackage{mathrsfs}
\usepackage{epstopdf}
\usepackage{lscape}
\usepackage{array}
\usepackage[utf8]{inputenc}
\usepackage{tikz,caption}
\usetikzlibrary{arrows.meta}
\usepackage{thmtools}
\usepackage{cleveref}

\usepackage{enumitem}
\setlist[itemize]{leftmargin=2em}
\setlist[enumerate]{leftmargin=2em}

\usepackage{booktabs}
\usepackage{multirow}
\usepackage{mathtools}
\usepackage[linesnumbered,ruled]{algorithm2e}
\usepackage{mathtools}
\usepackage{soul}
\usepackage{upgreek}

\newtheorem{theorem}{Theorem}[section]
\newtheorem{lemma}[theorem]{Lemma}
\newtheorem{corollary}[theorem]{Corollary}

\theoremstyle{definition}
\newtheorem{definition}[theorem]{Definition}

\newtheorem{example}[theorem]{Example}

\newtheorem{remark}[theorem]{Remark}

\numberwithin{theorem}{section}

\numberwithin{equation}{section}

\newcolumntype{C}[1]{>{\centering\arraybackslash}m{#1}}

\newcommand{\mb}[1]{\mathbf{#1}}

\newcommand{\grapheq}[2][1]{%
	\begin{tikzpicture}[baseline=0, vx/.style={circle, fill=black, minimum size=4pt, inner sep=0pt}, lc/.style = {circle, minimum size=6pt, inner sep=0pt, draw=black}, cat/.style ={rectangle, minimum size=10pt, inner sep=0pt, draw=black}, thick, scale=#1]
		#2
	\end{tikzpicture}%
}

\DeclarePairedDelimiter{\set}{\{}{\}}

\definecolor{darkblue}{rgb}{0.0,0,0.7} 
\definecolor{darkred}{rgb}{0.7,0,0} 
\definecolor{darkgreen}{rgb}{0, .6, 0} 

\newcommand{\defncolor}{\color{darkred}}
\newcommand{\defn}[1]{{\defncolor\emph{#1}}} 

\usepackage[colorinlistoftodos]{todonotes}

\title{Unicyclic Graphs of Arbitrary Girth with the Same Chromatic Symmetric Function}
\author[A.~Bingham]{Aram Bingham}\address[A.~Bingham]{Departamento de Matem\'aticas, Universidad de Chile, 3425 Las Palmeras, Ñuñoa, Santiago, CL}
\email{\textcolor{blue}{\href{mailto:aram@matmor.unam.mx}{aram@matmor.unam.mx}}}

\date{\today}

\begin{document}
	
	\begin{abstract}
		An open question asks whether the chromatic symmetric function (CSF) of a graph distinguishes non-isomorphic trees. While it is known that the CSF does not distinguish unicyclic graphs, examples of pairs of unicyclic graphs with the same CSF and girth larger than 3 were not known until very recently. This manuscript exhibits a sequence of pairs of non-isomorphic, connected, unicyclic graphs with increasing girth and which share the same CSF. This also provides the first infinite family of bipartite graphs with the same CSF. Our main technique is to apply a version of the ``triple deletion" modular relation, due independently to Guay-Paquet and Orellana--Scott.
	\end{abstract}
	
	\maketitle
	\section{Introduction}
	The chromatic symmetric function 
	$\mathbf{X}_G$ 
	of a simple graph $G$ was introduced in \cite{stanley1995} by Stanley.  
	For a graph $G$ with vertex set $V=\{v_1, \dots, v_n\}$,  $\mathbf{X}_G$ is defined as
	\[
	\mathbf{X}_G =
	\sum_\kappa x_{\kappa(v_1)}x_{\kappa(v_2)}\dots x_{\kappa(v_n)},
	\]
	where $\kappa: V \to \mathbb{N}$ ranges over all proper colorings of the vertex set.  The ``tree isomorphism problem," often attributed to Stanley, asks whether the chromatic symmetric function (CSF) distinguishes non-isomorphic trees. That is, for trees $T$, $T'$, is it the case that
	\[\mb{X}_T= \mb{X}_{T'} \implies T\cong T'?\]
	It has been verified computationally that this is the case for all trees on up to 29 vertices \cite{HeilJi18}. Other results identify specific families of trees which are distinguished by the CSF; see, for instance, \cite{AWvW24,ALISTEPRIETO,qCats,HurynChmutov20,LoeblSereni, martin2007distinguishingtreeschromaticsymmetric,WANG}. 
	
	In order to better understand what graphical information the CSF is capable of capturing, some authors have also studied the CSF of unicyclic graphs \cite{BJLOPS, DVW17,martin2007distinguishingtreeschromaticsymmetric, ORELLANA20141,  WW23}. It has been known since work of Orellana--Scott \cite{ORELLANA20141} that there are infinitely many pairs of non-isomorphic unicyclic graphs containing triangles which have the same CSF. More recently, the author together with Johnston, Lawson, Orellana, Pan and Sato exhibited the first examples of unicyclic graphs of girth 4 and 5 with the same CSF \cite{BJLOPS}. Based on computational evidence described in that reference, triangle-free examples of non-isomorphic pairs of graphs with the same CSF appear to be much rarer, further substantiating observations made in  \cite{APCSZ21} in which the first triangle-free (but neither unicyclic, nor bipartite) examples were described.
	
	Further responding to a question raised in \cite{APCSZ21}, in this article we exhibit an infinite sequence of pairs of non-isomorphic unicyclic graphs with the same CSF. As bipartite graphs are characterized as graphs with no odd cycles, the subsequence of our pairs of graphs with even cycle size provides infinitely many pairs of non-isomorophic bipartite graphs with the same CSF. Thus, our examples indicate that the bipartite property does not confer distinguishing power upon the CSF. Of course, all trees are bipartite. 
	
	Our sequence of pairs is based on the smallest pairs of girth 4 and girth 5 unicyclic graphs reported in \cite{BJLOPS}. There is only one other pair of unicyclic graphs with girth at least 4 and with the same CSF on graphs of up to 17 vertices. It is possible that this other pair (on 13 vertices and with a 4-cycle) is the first in another infinite sequence of increasing girth pairs with the same CSF. However, as the proof of the main result of this manuscript shows, we currently lack a clear understanding of the pattern to which such pairs belong.  
	
	Nevertheless, the existence of arbitrary girth pairs of unicyclic graphs with the same CSF offers some insight into the tree isomorphism problem. Our pairs of non-isomorphic unicyclic graphs differ from one another by only two edges. Through repeated application of a modular relation on graph CSFs \cite{guaypaquet,ORELLANA20141},\footnote{This four term relation (or one specific form of it) is also sometimes also called ``triple deletion," or the ``triangular relation."} we transform the CSF of each graph in a pair into identical expressions. This \emph{modular relation} has been shown by Penagui\~ao to generate (together with isomorphism relations on graphs) the kernel of a Hopf algebra map from graphs to symmetric functions, implying that it is sufficient to determine all linear relations between CSFs of non-isomorphic graphs \cite{penaguiao}.

	On the other hand, the lengthy sequence of calculations we have found necessary to show equality of CSFs indicates the subtlety of the modular relation. If examples of pairs of non-isomorphic trees with the same CSF exist, they may also differ by a small number of edges without being easily related by the modular relation. The challenge of finding the correct sequence of applications of this relation to prove equality of CSFs further indicates the complexity of CSF computation and identity testing.
	
	Despite significant experimental and theoretical work suggesting that the CSF distinguishes trees, we interpret our examples as suggesting that the possiblity it does not is worth consideration. As trees are graphs with infinite girth, a sequence of unicyclic pairs with increasing girth and the same CSF hews as close as is possible without actually answering the tree isomorphism question negatively. On the other hand, this would make it all the more remarkable if it turns out to be the case that the CSF does distinguish trees. 
	
	Part of the attractiveness of the tree isomorphism problem lies in the well-known fact that all trees on $n$ vertices have the same \emph{chromatic polynomial}, $\chi(k)=k(k-1)^{n-1}$. It is also true that all \emph{$c$-unicyclic graphs} (those with unique cycle of order $c$) on $n$ vertices also have the same chromatic polynomial: \[\chi(k)=(k-1)^n+(-1)^c(k-1)^{n-c+1}.\] Noting that the only known pairs of non-ismorphic unicyclic graphs with a triangle occur for even $n$, and in view of the scarcity of known pairs for $c>3$, we close by remarking that it seems there may be arithmetic relations between $n$ and $c$ which determine (to some extent) the distinguishing power of the CSF among unicyclic graphs.

	
	This paper is organized as follows. In Section~\ref{sec.def} we recall a version of the modular relation and outline its application. Next, in Section~\ref{sec.const} we describe our two sequences of unicyclic graphs which form the pairs with identical CSF. Section~\ref{sec.proof} is devoted to the calculations involved in verifying that the corresponding members of each sequence of graphs give the same CSF.  
	
	\subsection*{Acknowledgments}The author is supported by FONDECYT-ANID grant 3250472.
	
	\section{Definitions and preliminaries}
	\label{sec.def}
	
	All our graphs are simple. We recall here a version of the modular relation from \cite[Corollary 3.2]{ORELLANA20141}.  For a graph $G=(V,E)$ with $v_1,v_2\in V$ and $e=\set{v_1,v_2}$, we write $G+\set{v_1,v_2}$ or $G+e$ to denote the graph with vertex set $V$ and edge set $E\cup \set{\set{v_1,v_2}}$. Similarly, write $G-\set{v_1,v_2}$ or $G-e$ for the usual edge deletion operation. Now let graph $G$ contain incident edges $e_1=\set{u,v}$ and $e_2=\set{v,w}$, with $e_3=\set{u,w}$. Then the modular relation yields
	\begin{equation}
		\label{eq:triplerelation}
		\mb{X}_G=\mb{X}_{G-e_1+e_3}+\mb{X}_{G-e_2}-\mb{X}_{G-e_1-e_2+e_3}.
	\end{equation}
	Of course, one obtains another version of the relation by interchanging $e_1$ and $e_2$.
	
	Since the computations in this article will involve identities only between CSFs of graphs, we will use the symbol $\doteq$ to denote an equality of the CSFs of the indexing graphs. That is, instead of \eqref{eq:triplerelation}, we will write equations (with appropriate parentheses) of the form 
	\begin{equation}\label{eq:dottriple}
		G\doteq {(G-e_1+e_3)}+{(G-e_2)}-{(G-e_1-e_2+e_3)}.
	\end{equation}
	
	We sometimes find it easier to depict and apply this relation visually as follows. Consider the graph $G$ with $u,v,w$ as its only vertices, and $e_1,e_2,e_3$ as above. Then \eqref{eq:dottriple} becomes 
	\begin{equation}
		\begin{tikzpicture}[auto=center,every node/.style={circle, fill=black, scale=0.45}, style=thick, scale=0.4] \label{tikz:triple1}
			\node (A1) at (0,0) {};
			\filldraw[black] (1, 2) coordinate (A2) circle (4pt) node{};
			\node (A3) at (2,0) {};
			\draw(A1) -- (A2);
			\draw(A2) -- (A3);
			\draw[color=red](A1) -- (A2) node[above, left = 6pt,fill=white] {\huge{$\mb{e_1}$}};
			\draw[color=YellowGreen](A2) -- (A3) node[above=12pt, fill=white] {\huge{$\mb{e_2}$}};
			\node (eq) at (4,1) [fill=white] {\Huge{$\doteq$}};
			\node (B1) at (6,0) {};
			\filldraw[black] (7, 2) coordinate (B2) circle (4pt) node{};
			\node (B3) at (8,0) {};
			\draw(B1) -- (B3);
			\draw(B2) -- (B3);
			\draw[color=cyan](B1) -- (B3) node[below left = 6pt,fill=white] {\huge{$\mb{e_3}$}};
			\draw[color=YellowGreen](B2) -- (B3) node[above=10pt,fill=white] {\huge{$\mb{e_2}$}};
			\node (eq) at (10,1) [fill=white] {\Huge{$+$}};
			\node (C1) at (12,0) {};
			\filldraw[black] (13, 2) coordinate (C2) circle (4pt) node{};
			\node (C3) at (14,0) {};
			\draw[color=red](C1) -- (C2) node[left = 6pt,fill=white] {\huge{$\mb{e_1}$}};
			\node (eq) at (16,1) [fill=white] {\Huge{$-$}};
			\node (D1) at (18,0) {};
			\filldraw[black] (19, 2) coordinate (D2) circle (4pt) node{};
			\node (D3) at (20,0) {};
			\draw[color=cyan](D1) -- (D3) node[below left = 6pt,fill=white] {\huge{$\mb{e_3}$}};
			
			\node at (21,0) [fill=white] {\huge{.}};
		\end{tikzpicture}
	\end{equation}
	We use this color scheme when helpful to visually depict equations of the form \eqref{eq:dottriple}. This is illustrated in the next example.
	
	\begin{example}
		Let $G$ be the unicyclic graph on 12 vertices depicted below.
		
		
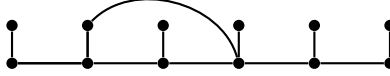
\begin{center}
			\begin{tikzpicture}[auto=center,every node/.style={circle, fill=black, scale=0.45}, style=thick, scale=0.50] 
				\node (S1) at (0,0) {};
				\node (S2) at (2,0) {};
				\node (S3) at (4,0) {};
				\node (S4) at (6,0) {};
				\node (S5) at (8,0) {};
				\node (S6) at (10,0) {};
				\draw(S1) -- (S2);
				
				\foreach \x [count=\i] in {0,2,4,6,8,10} {
					\node (L\i) at (\x,1) {};
					\draw (S\i) -- (L\i);
				}
				\foreach \i in {1,...,5} {
					\draw (S\i) -- (S\the\numexpr\i+1\relax);
				}
				\draw[thick,  bend right=60] (S4) to (L2);
				\draw[thick] (L2) to (S2);
			\end{tikzpicture}
		\end{center}
		Identifying two incident edges of the graph as $e_1$ and $e_2$, we have the equation below.
		\begin{equation}\label{eq:G4ex}
			\begin{tikzpicture}[auto=center,every node/.style={circle, fill=black, scale=0.45}, style=thick, scale=0.30] 
				\node (S1) at (0,0) {};
				\node (S2) at (2,0) {};
				\node (S3) at (4,0) {};
				\node (S4) at (6,0) {};
				\node (S5) at (8,0) {};
				\node (S6) at (10,0) {};
				\draw(S1) -- (S2);
				
				\foreach \x [count=\i] in {0,2,4,6,8,10} {
					\node (L\i) at (\x,1) {};
					\draw (S\i) -- (L\i);
				}
				\foreach \i in {1,...,5} {
					\draw (S\i) -- (S\the\numexpr\i+1\relax);
				}
				\draw[thick, red, bend right=60] (S4) to (L2);
				\draw[very thick, YellowGreen] (L2) to (S2);
				\node [fill=white] at (12,0.5) {\Huge{$\doteq$}};
				\foreach \i [evaluate=\i as \j using int(2*\i+12)] in {1,...,6} {
					\node (sa\i) at (\j,0) {};
					\node (la\i) at (\j,1) {};
					\draw (sa\i) -- (la\i);
				}
				\foreach \i in {1,...,5} {
					\draw (sa\i) -- (sa\the\numexpr\i+1\relax);
				}
				\draw[very thick, YellowGreen] (sa2) to (la2);
				\draw[thick, cyan,  bend left=60] (sa4) to (sa2);
				\node [fill=white] at (26,0.5) {\Huge{$+$}};
				\foreach \i [evaluate=\i as \j using int(2*\i+26)] in {1,...,6} {
					\node (sb\i) at (\j,0) {};
					\node (lb\i) at (\j,1) {};
				}
				\foreach \i [evaluate=\i as \j using int(2*\i+26)] in {1,3,4,5,6}{
					\draw (sb\i) -- (lb\i);
				}
				\foreach \i in {1,...,5} {
					\draw (sb\i) -- (sb\the\numexpr\i+1\relax);
				}
				\draw[thick,white] (sb2) to (lb2);
				\draw[thick, red,  bend right=60] (sb4) to (lb2);
				\node [fill=white] at (40,0.5) {\Huge{$-$}};
				\foreach \i [evaluate=\i as \j using int(2*\i+40)] in {1,3,4,5,6} {
					\node (sc\i) at (\j,0) {};
					\node (lc\i) at (\j,1) {};
				}
				\node (sc2) at (44,0) {};
				\node (lc2) at (44,1) {};
				\foreach \i in {1,...,5} {
					\draw (sc\i) -- (sc\the\numexpr\i+1\relax);
				}
				\foreach \i [evaluate=\i as \j using int(2*\i+40)] in {1,3,4,5,6}{
					\draw (sc\i) -- (lc\i);
				}
				\draw[thick, cyan,  bend left=60] (sc4) to (sc2);
			\end{tikzpicture}
		\end{equation}
	\end{example}
	
	As many of the graphs that appear in our computations are \defn{caterpillars} or forests of caterpillars, we now introduce some notation to represent these compactly. Caterpillars are trees such that the induced subgraph on all non-leaf vertices is just a path. The order (number of vertices) $\ell$ of this path is referred to as the \defn{length} of the caterpillar. This path subgraph is also called the \defn{spine} of the caterpillar.
	
	Non-leaf vertices of a graph are called \defn{internal vertices}. For each internal vertex $v$, the subgraph induced by $v$ together with all its leaf neighbors is called a \defn{leaf component}. For a caterpillar $G=(V,E)$, we label the spine vertices $v_1,\dots,v_\ell$ such that $\set{v_i, v_{i+1}}\in E$ for each $1\leq i \leq \ell-1$. Taking $\alpha_i$ as the order of the leaf component of $v_i$, a caterpillar may be identified by an integer composition $\alpha=(\alpha_1,\dots,\alpha_\ell)$. This composition is an invariant of the graph, unique up to reversal of the composition; that is $\alpha$ and $\alpha^*:=(\alpha_\ell,\dots,\alpha_1)$ correspond to the same isomorphism class of caterpillar. We will write $[\alpha]:=\set{\alpha,\alpha^*}$ to denote the equivalence class consisting of a composition and its reversal. As an example, the caterpillar which appears as the middle term on the right-hand side of \eqref{eq:G4ex} corresponds to the composition class $[(2,1,2,3,2,2)]$. 
	
	When referring to a caterpillar $T$, we will sometimes write $T=[\alpha]$ to mean that $T$ is the caterpillar indexed by composition class $[\alpha]$.  For instance, a caterpillar of the form $[(2^k)]$ for some $k\geq 3$ is called a \defn{comb} graph.

	Forests of caterpillars will frequently appear in our calculations. For compositions $\alpha,\beta$, we will write $[\alpha]\sqcup[\beta]$ to refer to the (CSF of the) graph which consists of the disjoint union of the two caterpillars indexed by those composition classes.
	
 Note that the only restrictions on the parts of a composition identifying a caterpillar are that $\alpha_1,\alpha_\ell\geq 2$. To make our notation for compositions more concise, we may also write
	\[(\alpha_1,\dots,\overbrace{m,\dots,m}^{\text{$k$ times}},\dots, \alpha_\ell) =(\alpha_1,\dots,m^k,\dots,\alpha_\ell).\]
	For example, $(2,3,2,2,2,2,4,3,2)=(2,3,2^4,4,3,2)$ as compositions. This should not provoke confusion as we will perform no multiplication or exponentiation operations on the entries of our compositions.

	\section{Construction}
	\label{sec.const}
	We will denote our sequence of pairs of non-isomorphic unicyclic graphs of increasing girth as $(G_c,H_c)$ where $c\geq 4$ indicates the size of the unique cycle.  Consider first the disjoint union of two comb graphs: $[(2^{c-1})]\sqcup[(2^{c-1})]$. Identify this pair of combs as consisting of a ``left" and a ``right" comb and label the spine vertices $\{v_1^L,\dots, v_{c-1}^L\}$ and $\{v_1^R,\dots,v_{c-1}^R\}$ respectively. Similarly label the corresponding leaf vertices $(f_1^L,\dots, f_{c-1}^L)$ and $(f_1^R,\dots,f_{c-1}^R)$. For example, the graph $[(2^3)]\sqcup [(2^3)]$ is depicted as labelled below.
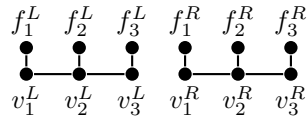
\begin{center}
	\begin{tikzpicture}
	[auto=center,vx/.style={circle, fill=black, minimum size=5pt, inner sep=0pt}, style=thick,  scale=0.7] 
	\foreach \i in {1,...,6} {
		\node[vx] (S\i) at (\i,0) {};
		\node[vx] (L\i) at (\i,0.5) {};
		\draw (S\i) -- (L\i);
	}
	\foreach \i in {1,2,4,5} {
		\draw (S\i) -- (S\the\numexpr\i+1\relax);
	}
	\foreach \i [evaluate=\i as \j using int(\i+3)] in {1,2,3}  {
	\node at (\i,-0.5) {$v_\i^L$};
	\node at (\j,-0.5) {$v_\i^R$};
	\node at (\i,1) {$f_\i^L$};
	\node at (\j,1) {$f_\i^R$};
	}
	\end{tikzpicture}
\end{center}

\begin{definition}
	For any $c\geq 4$, the pair of graphs $(G_c,H_c)$ on $4c-4$ vertices are constructed from the union $[(2^{c-1})]\sqcup [(2^{c-1})]$ with vertices as labeled as above, and subject to the following modifications:
	\begin{itemize}
	\item add edges $\set{v_{c-1}^L, v_1^R}$ and $\set{f_2^L,v_1^R}$ to obtain $G_c$.
	\item add edges $\set{v_{c-1}^L, f_1^R}$ and $\set{v_2^L,v_1^R}$ to obtain $H_c$.	
	\end{itemize}	
\end{definition}

The pair $(G_4,H_4)$ is depicted below.
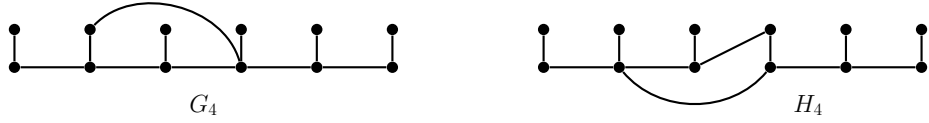
\begin{center}
	\begin{tikzpicture}[auto=center,every node/.style={circle, fill=black, scale=0.45}, style=thick, scale=0.50] 
		\foreach \i [evaluate=\i as \j using int(2*\i)] in {1,...,6} {
			\node (S\i) at (\j,0) {};
			\node (L\i) at (\j,1) {};
			\draw (S\i) -- (L\i);
		}
		\foreach \i in {1,...,5} {
			\draw (S\i) -- (S\the\numexpr\i+1\relax);
		}
		\draw[thick,  bend right=60] (S4) to (L2);
		\draw[thick] (L2) to (S2);
		\node[fill=white] at (7,-1) {\Huge{$G_4$}};
		\foreach \i [evaluate=\i as \j using int(2*\i+14)] in {1,...,6} {
			\node (sb\i) at (\j,0) {};
			\node (lb\i) at (\j,1) {};
			\draw (sb\i) -- (lb\i);
		}
		\foreach \i in {1,2,4,5} {
			\draw (sb\i) -- (sb\the\numexpr\i+1\relax);
		}
		\draw (sb3) -- (lb4);
		\draw [bend right =50] (sb2) to (sb4);
		\node[fill=white] at (23,-1) {\Huge{$H_4$}};
	\end{tikzpicture}
\end{center}
As per our definition, $G_4$ can be constructed from $[(2^3)]\sqcup [(2^3)]$ by adding edges $\set{v_3^L, v_1^R}$ and $\set{f_2^L,v_1^R}$, while $H_4$ is constructed by adding edges $\set{v_3^L, f_1^R}$ and $\set{v_2^L,v_1^R}$. Note that these graphs are the same pair that appears in \cite[Figure 14]{BJLOPS}. The pair $(G_5, H_5)$ also appear in that reference as Figure 16. For convenience, we illustrate $(G_5,H_5)$ and $(G_6,H_6)$ below. 
		
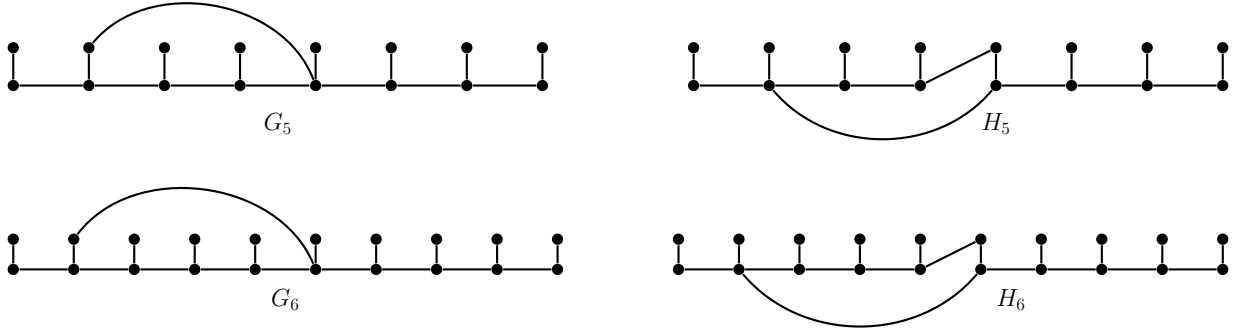
\begin{center}
		\begin{tikzpicture}[auto=center,every node/.style={circle, fill=black, scale=0.45}, style=thick, scale=0.50] 
			\foreach \i [evaluate=\i as \j using int(2*\i)] in {1,...,8} {
				\node (S\i) at (\j,0) {};
				\node (L\i) at (\j,1) {};
				\draw (S\i) -- (L\i);
			}
			\foreach \i in {1,...,7} {
				\draw (S\i) -- (S\the\numexpr\i+1\relax);
			}
			\draw[thick,  bend right=60] (S5) to (L2);
			\node[fill=white] at (9,-1) {\Huge{$G_5$}};
			\foreach \i [evaluate=\i as \j using int(2*\i+18)] in {1,...,8} {
				\node (sb\i) at (\j,0) {};
				\node (lb\i) at (\j,1) {};
				\draw (sb\i) -- (lb\i);
			}
			\foreach \i in {1,2,3,5,6,7} {
				\draw (sb\i) -- (sb\the\numexpr\i+1\relax);
			}
			\draw (sb4) -- (lb5);
			\draw [bend right =50] (sb2) to (sb5);
			\node[fill=white] at (28,-1) {\Huge{$H_5$}};
		\end{tikzpicture}
		\begin{tikzpicture}[auto=center,every node/.style={circle, fill=black, scale=0.45}, style=thick, scale=0.4] 
			\foreach \i [evaluate=\i as \j using int(2*\i)] in {1,...,10} {
				\node (S\i) at (\j,0) {};
				\node (L\i) at (\j,1) {};
				\draw (S\i) -- (L\i);
			}
			\foreach \i in {1,...,9} {
				\draw (S\i) -- (S\the\numexpr\i+1\relax);
			}
			\draw[thick,  bend right=60] (S6) to (L2);
			\node[fill=white] at (11,-1) {\Huge{$G_6$}};
			\foreach \i [evaluate=\i as \j using int(2*\i+22)] in {1,...,10} {
				\node (sb\i) at (\j,0) {};
				\node (lb\i) at (\j,1) {};
				\draw (sb\i) -- (lb\i);
			}
			\foreach \i in {1,2,3,4,6,7,8,9} {
				\draw (sb\i) -- (sb\the\numexpr\i+1\relax);
			}
			\draw (sb5) -- (lb6);
			\draw [bend right =50] (sb2) to (sb6);
			\node[fill=white] at (35,-1) {\Huge{$H_6$}};
		\end{tikzpicture}
	\end{center}
	
	\begin{lemma}
		For each $c\geq 4$, $G_c$ and $H_c$ are not isomorphic.
	\end{lemma}
	\begin{proof}
		For given $c$, label the vertices of $G_c$ and $H_c$ by those given to the two combs used to construct each graph. Each unicyclic graph has a unique cycle vertex of degree 4 and a unique cycle vertex of degree 2.  These are $v_1^R$ and $f_2^L$ for $G_c$, and $v_2^L$ and $f_1^R$ for $H_c$, respectively. For $c\geq 4$, the pair of vertices are adjacent in $G_c$, but not in $H_c$.
	\end{proof}
	
	Our formula for the identical CSF of these non-isomorphic pairs will consist of (the CSFs of) two other unicyclic graphs and a number of other special caterpillar forests. We now introduce those graphs. 
	
	\begin{definition}
		Let $A_c$ and $B_c$ be the unicyclic graphs obtained from the union $[(2^{c-1})]\sqcup [(2^{c-1})]$ with vertices as labeled above subject to the following modifications:
		\begin{itemize}
			\item Add edge $\set{v_2^L, v_1^R}$ and $\set{v_c^L,v_1^R}$ to obtain $A_c$.
			\item Starting from $A_c$, delete edge $\set{v_2^L,f_2^L}$ to obtain $B_c$.
		\end{itemize}
	\end{definition}
We illustrate $A_5$ and $B_5$ below. 
	
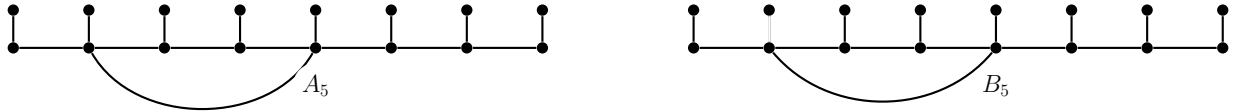
\begin{center}
	\begin{tikzpicture}[auto=center,every node/.style={circle, fill=black, scale=0.45}, style=thick, scale=0.50] 
		\foreach \i [evaluate=\i as \j using int(2*\i)] in {1,...,8} {
			\node (S\i) at (\j,0) {};
			\node (L\i) at (\j,1) {};
			\draw (S\i) -- (L\i);
		}
		\foreach \i in {1,...,7} {
			\draw (S\i) -- (S\the\numexpr\i+1\relax);
		}
		\draw[thick,  bend left=60] (S5) to (S2);
		\node[fill=white] at (10,-1) {\Huge{$A_5$}};
		\foreach \i [evaluate=\i as \j using int(2*\i+18)] in {1,...,8} {
			\node (sb\i) at (\j,0) {};
			\node (lb\i) at (\j,1) {};
			\draw (sb\i) -- (lb\i);
		}
		\foreach \i in {1,2,3,4,5,6,7} {
			\draw (sb\i) -- (sb\the\numexpr\i+1\relax);
		}
		\draw [color=white] (sb2) -- (lb2);
		\draw [bend right =50] (sb2) to (sb5);
		\node[fill=white] at (28,-1) {\Huge{$B_5$}};
	\end{tikzpicture}
	\end{center}
	The common expression for the CSFs of our pairs will involve one other type of caterpillar, and then many graphs which are the disjoint union of two caterpillars. These graphs, like $(A_c, B_c)$, come in pairs with opposite sign. We now describe this special caterpillar and the next pair that appear in the expression for $\mb{X}_{G_c}=\mb{X}_{H_c}$.
	
	\begin{definition}
		Define $T_c$ as the caterpillar on $4c-4$ vertices which corresponds to the composition class $[(2^{c-2},1,3,2^{c-2})]$. Define $S_c$ as $[(2^{c-2})]\sqcup [(4,2^{c-2})]$. Finally, define $F_c$ as $[(3)]\sqcup [(2^{c-3},3,2^{c-2})]$
	\end{definition}
	For example, $T_5$ and $S_5$ are depicted below.
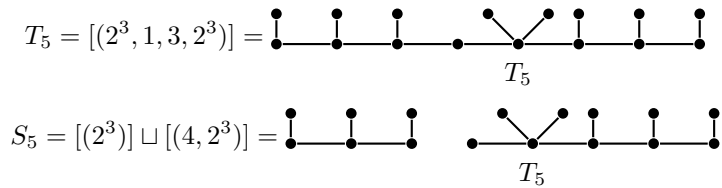
\[ T_5=[(2^3,1,3,2^3)]=
    \grapheq[0.4]{	
	\foreach \i [evaluate=\i as \j using int(2*\i)] in {1,...,8} {
		\node[vx] (S\i) at (\j,0) {};
	}
	\foreach \i [evaluate=\i as \j using int(2*\i)] in {1,2,3,6,7,8}{
		\node[vx] (L\i) at (\j,1) {};
		\draw (S\i) -- (L\i);
	}
	\foreach \i in {1,...,7} {
		\draw (S\i) -- (S\the\numexpr\i+1\relax);
	}
	\node[vx] (L2) at (9,1) {};
 	\node[vx] (L5) at (11,1) {};
	\draw[thick] (S5) to (L2);
	\draw[thick] (S5) to (L5);
	\node[fill=white] at (10,-1) {$T_5$};
	} 
\]
\[ S_5=[(2^3)]\sqcup[(4,2^3)]=
\grapheq[0.4]{	
	\foreach \i [evaluate=\i as \j using int(2*\i)] in {1,...,8} {
		\node[vx] (S\i) at (\j,0) {};
	}
	\foreach \i [evaluate=\i as \j using int(2*\i)] in {1,2,3,6,7,8}{
		\node[vx] (L\i) at (\j,1) {};
		\draw (S\i) -- (L\i);
	}
	\foreach \i in {1,2,4,5,6,7} {
		\draw (S\i) -- (S\the\numexpr\i+1\relax);
	}
	\node[vx] (L2) at (9,1) {};
	\node[vx] (L5) at (11,1) {};
	\draw[thick] (S5) to (L2);
	\draw[thick] (S5) to (L5);
	\node[fill=white] at (10,-1) {$T_5$};
} 
\]
\begin{example}\label{ex:G4}
	We illustrate the calculation we will make in general for $G_c$ with $G_4$. Applying the modular relation \eqref{eq:dottriple} starting from \eqref{eq:G4ex}, we have
	\begin{align*}
		\grapheq[0.3]{
		\foreach \x [count=\i] in {0,2,4,6,8,10} {
			\node[vx] (S\i) at (\x,0) {};
			\node[vx] (L\i) at (\x,1) {};
			\draw (S\i) -- (L\i);
		}
		\foreach \i in {1,...,5} {
			\draw (S\i) -- (S\the\numexpr\i+1\relax);
		}
		\draw[bend right=60] (S4) to (L2);
		}
		&\doteq	\grapheq[0.3]{
			\foreach \x [count=\i] in {0,2,4,6,8,10} {
				\node[vx] (S\i) at (\x,0) {};
				\node[vx] (L\i) at (\x,1) {};
				\draw (S\i) -- (L\i);
			}
			\foreach \i in {1,...,5} {
				\draw (S\i) -- (S\the\numexpr\i+1\relax);
			}
			\draw[bend left=60] (S4) to (S2);
		}
		+\grapheq[0.3]{
		\foreach \x [count=\i] in {0,2,4,6,8,10} {
			\node[vx] (S\i) at (\x,0) {};
		}
		\node[vx] (L1) at (0,1) {};
		\node[vx] (L2) at (5.5,1) {};
		\node[vx] (L3) at (4,1) {};
		\node[vx] (L4) at (6.5,1) {};
		\node[vx] (L5) at (8,1) {};
		\node[vx] (L6) at (10,1) {};
		\draw (S1) -- (L1);
		\draw (S3) -- (L3);
		\draw (S4) -- (L2);		
		\draw (S4) -- (L4);
		\draw (S5) -- (L5);
		\draw (S6) -- (L6);
		\foreach \i in {1,...,5} {
			\draw (S\i) -- (S\the\numexpr\i+1\relax);
		}
		\draw[very thick,  YellowGreen] (S2) to (S3);
		\draw[very thick,  red] (S3) -- (L3);
		}
		-\grapheq[0.3]{
			\foreach \x [count=\i] in {0,2,4,6,8,10} {
				\node[vx] (S\i) at (\x,0) {};
				\node[vx] (L\i) at (\x,1) {};
				\draw (S\i) -- (L\i);
			}
			\foreach \i in {1,...,5} {
				\draw (S\i) -- (S\the\numexpr\i+1\relax);
			}
			\draw[bend left=60] (S4) to (S2);
			\draw[very thick,  white] (L2) -- (S2);
		} \\ 
		&\doteq A_4-B_4+(\grapheq[0.3]{
			\foreach \x [count=\i] in {0,2,4,6,8,10} {
				\node[vx] (S\i) at (\x,0) {};
			}
			\node[vx] (L1) at (0,1) {};
			\node[vx] (L2) at (2,1) {};
			\node[vx] (L3) at (5.5,1) {};
			\node[vx] (L4) at (6.5,1) {};
			\node[vx] (L5) at (8,1) {};
			\node[vx] (L6) at (10,1) {};
			\draw (S1) -- (L1);
			\draw (S2) -- (L2);
			\draw (S4) -- (L3);		
			\draw (S4) -- (L4);
			\draw (S5) -- (L5);
			\draw (S6) -- (L6);
			\foreach \i in {1,...,5} {
				\draw (S\i) -- (S\the\numexpr\i+1\relax);
			}
			\draw[very thick,  cyan] (S2) to (L2);
			\draw[very thick,  YellowGreen] (S2) -- (S3);
		}+
		(\grapheq[0.3]{
			\foreach \x [count=\i] in {0,2,4,6,8,10} {
				\node[vx] (S\i) at (\x,0) {};
			}
			\node[vx] (L1) at (0,1) {};
			\node[vx] (L2) at (4,1) {};
			\node[vx] (L3) at (5.5,1) {};
			\node[vx] (L4) at (6.5,1) {};
			\node[vx] (L5) at (8,1) {};
			\node[vx] (L6) at (10,1) {};
			\draw (S1) -- (L1);
			\draw (S3) -- (L2);
			\draw (S4) -- (L3);		
			\draw (S4) -- (L4);
			\draw (S5) -- (L5);
			\draw (S6) -- (L6);
			\foreach \i in {1,...,5} {
				\draw (S\i) -- (S\the\numexpr\i+1\relax);
			}
			\draw[very thick,  white] (S2) to (S3);
			\draw[very thick,  red] (S3) -- (L2);
		})
		-
		\grapheq[0.3]{
			\foreach \x [count=\i] in {0,2,4,6,8,10} {
				\node[vx] (S\i) at (\x,0) {};
			}
			\node[vx] (L1) at (0,1) {};
			\node[vx] (L2) at (2,1) {};
			\node[vx] (L3) at (5.5,1) {};
			\node[vx] (L4) at (6.5,1) {};
			\node[vx] (L5) at (8,1) {};
			\node[vx] (L6) at (10,1) {};
			\draw (S1) -- (L1);
			\draw (S2) -- (L2);
			\draw (S4) -- (L3);		
			\draw (S4) -- (L4);
			\draw (S5) -- (L5);
			\draw (S6) -- (L6);
			\foreach \i in {1,...,5} {
				\draw (S\i) -- (S\the\numexpr\i+1\relax);
			}
			\draw[very thick, cyan] (S2) to (L2);
			\draw[very thick, white] (S2) -- (S3);
		}
		)\\
		&\doteq A_4-B_4+T_4-S_4+F_4.
	\end{align*}
\end{example}

In what follows we will also use the notation $\alpha \cdot \beta$ to denote the \defn{concatenation} operation on compositions. That is, if $\alpha=(\alpha_1,\dots,\alpha_l)$ and $\beta=(\beta_1,\dots,\beta_m)$, then $\alpha\cdot \beta=(\alpha_1,\dots,\alpha_l,\beta_1,\dots,\beta_m)$. We will write $\alpha \odot \beta$ to denote the \defn{near-concatenation} operation on compositions. That is $\alpha \odot \beta=(\alpha_1,\dots,\alpha_{l-1},\alpha_l+\beta_1,\beta_2,\dots,\beta_m)$.

\section{Proof}
	\label{sec.proof}
First we present our formula for $\mb{X}_{G_c}.$ This will be easier to establish than $\mb{X}_{H_c}$ since it arises as a cancellation-free computation with straightforward application of the triple deletion relation. We start with a lemma which tells us how to ``shuffle" legs (leaves) from one side of our caterpillars to the other. 
\begin{lemma}\label{lem:shuffle1}
	For a caterpillar with composition of the form $[\alpha\cdot(2,1,2)\cdot \beta]$ for any $\alpha,\beta$, we have
	\[ [\alpha\cdot(2,1,2)\cdot \beta]\doteq [\alpha\cdot(2,2,1)\cdot\beta]+[\alpha\cdot(3)\sqcup(2)\cdot\beta]-[\alpha\cdot(2,2)\sqcup(1)\odot\beta].\]
\end{lemma}
\begin{proof}
	Depict the caterpillar $[\alpha \cdot (2,1,2)\cdot \beta]$ as 
	
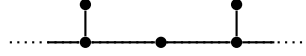
\begin{center}
		\begin{tikzpicture}[auto=center,every node/.style={circle, fill=black, scale=0.45}, style=thick, scale=0.50] 
		\foreach \i [evaluate=\i as \j using int(-4+2*\i)] in {1,2,3} {
		\node (S\i) at (\j,0) {};
	}
	\node (L1) at (-2,1) {};
	\node (L3) at (2,1) {};
	\draw (L1) -- (S1);
	\draw (-3,0) -- (S1) -- (S2) -- (S3) -- (3,0);
	\draw (S3) -- (L3);
	\draw[dotted] (-4,0) -- (4,0);
		\end{tikzpicture}
	\end{center}
	where the dotted lines represent the continuation into the arbitrary compositions $\alpha$ and $\beta$ on either side. Then, applying \eqref{eq:dottriple}, we have 
	\[\grapheq[0.3]{
			\draw[dotted] (-4,0) -- (4,0);
			\foreach \i [evaluate=\i as \j using int(-4+2*\i)] in {1,2,3} {
		\node[vx] (S\i) at (\j,0) {};
	}
	\node[vx] (L1) at (-2,1) {};
	\node[vx] (L3) at (2,1) {};
	\draw (L1) -- (S1);
	\draw (-3,0) -- (S1) -- (S2) -- (S3) -- (3,0);
	\draw (S3) -- (L3);
	\draw[color=YellowGreen, very thick] (S2)--(S3);
	\draw[color=red, very thick] (S3)--(L3);
}
	\doteq \grapheq[0.3]{\foreach \i [evaluate=\i as \j using int(-4+2*\i)] in {1,2,3} {
			\node[vx] (S\i) at (\j,0) {};
		}
		\node[vx] (L1) at (-2,1) {};
		\node[vx] (L3) at (2,1) {};
		\draw (L1) -- (S1);
		\draw (-3,0) -- (S1) -- (S2) -- (S3) -- (3,0);
		\draw[dotted] (-4,0) -- (4,0);
		\draw[color=YellowGreen, very thick] (S2)--(S3);
		\draw[color=cyan, very thick] (S2)--(L3);
		} 
		+
		\grapheq[0.3]{\foreach \i [evaluate=\i as \j using int(-4+2*\i)] in {1,2,3} {
					\node[vx] (S\i) at (\j,0) {};
				}
				\node[vx] (L1) at (-2,1) {};
				\node[vx] (L3) at (2,1) {};
				\draw (L1) -- (S1);
				\draw (-3,0) -- (S1) -- (S2);
				\draw (S3) -- (3,0);
				\draw[red, very thick] (S3) -- (L3);
				\draw[dotted] (-4,0) -- (-3,0);
				\draw[dotted] (3,0) -- (4,0);
				}
				-
				\grapheq[0.3]{\foreach \i [evaluate=\i as \j using int(-4+2*\i)] in {1,2,3} {
							\node[vx] (S\i) at (\j,0) {};
						}
						\node[vx] (L1) at (-2,1) {};
						\node[vx] (L3) at (2,1) {};
						\draw (L1) -- (S1);
						\draw (-3,0) -- (S1) -- (S2);
						\draw (S3) -- (3,0);
						\draw[very thick, cyan] (S2) -- (L3);
						\draw[dotted] (-4,0) -- (-3,0);
						\draw[dotted] (3,0) -- (4,0);}
	.\]
	The claimed formula is evident from the above equation.
\end{proof}
We will also need the following later on.
\begin{lemma}\label{lem:tripleDNC}
	For a caterpillar with composition of the form $[\alpha\cdot(2,1,2)\cdot \beta]$ for any $\alpha,\beta$, we have
	\[ [\alpha\cdot(2,1,2)\cdot \beta]\doteq [\alpha\cdot(3,2)\cdot\beta]+[\alpha\cdot(2)]\sqcup[(3)\cdot\beta]-[(1)]\sqcup[\alpha\cdot(2,2)\cdot\beta].\]
\end{lemma}
\begin{proof}
	 This time, apply \eqref{eq:dottriple} as 
	\[\grapheq[0.3]{
		\draw[dotted] (-4,0) -- (4,0);
		\foreach \i [evaluate=\i as \j using int(-4+2*\i)] in {1,2,3} {
			\node[vx] (S\i) at (\j,0) {};
		}
		\node[vx] (L1) at (-2,1) {};
		\node[vx] (L3) at (2,1) {};
		\draw (L1) -- (S1);
		\draw (-3,0) -- (S1) -- (S2) -- (S3) -- (3,0);
		\draw (S3) -- (L3);
		\draw[very thick, color=YellowGreen] (S1)--(S2);
		\draw[very thick, color=red] (S3)--(S2);
	}
	\doteq \grapheq[0.3]{\foreach \i [evaluate=\i as \j using int(-3+2*\i)] in {1,2} {
			\node[vx] (S\i) at (\j,0) {};
		}
		\node[vx] (L1) at (-1,1) {};
		\node[vx] (L2) at (0,1) {};
		\node[vx] (L3) at (1,1) {};
		\draw (L1) -- (S1);
		\draw (-2,0) -- (S1) -- (S2) -- (S3) -- (2,0);
		\draw[dotted] (-3,0) -- (3,0);
		\draw[very thick, color=YellowGreen] (S1)--(L2);
		\draw[very thick, color=cyan] (S2)--(S1);
		\draw (S2) -- (L3);
	}  
	+
	\grapheq[0.3]{\foreach \i [evaluate=\i as \j using int(-4+2*\i)] in {1,2,3} {
			\node[vx] (S\i) at (\j,0) {};
		}
		\node[vx] (L1) at (-2,1) {};
		\node[vx] (L3) at (2,1) {};
		\draw (L1) -- (S1);
		\draw (-3,0) -- (S1);
		\draw (S3) -- (3,0);
		\draw[very thick, red] (S3) -- (S2);
		\draw (S3) -- (L3);
		\draw[dotted] (-4,0) -- (-3,0);
		\draw[dotted] (3,0) -- (4,0);
	}
	-
	\grapheq[0.3]{
		\node[vx] (S1) at (-2,0) {};
		\node[vx] (S3) at (2,0) {};
		\node[vx] (L1) at (-2,1) {};
		\node[vx] (L2) at (0,1) {};
		\node[vx] (L3) at (2,1) {};
		\draw (L1) -- (S1);
		\draw (-3,0) -- (S1);
		\draw (S3) -- (3,0);
		\draw[very thick, cyan] (S3) -- (S1);
		\draw (S3) -- (L3);
		\draw[dotted] (-4,0) -- (-3,0);
		\draw[dotted] (3,0) -- (4,0);}
	.\]
\end{proof}
\begin{remark}\label{rk:genDNC}
	The application of the modular relation in \Cref{lem:tripleDNC} is equivalent to an application of the \emph{deletion-near-contraction} relation of \cite{AMOZ23}. The lemma further generalizes to 
	\[ [\alpha\cdot(a,1,b)\cdot \beta]\doteq [\alpha\cdot(a+1,b)\cdot\beta]+[\alpha\cdot(a)]\sqcup[(b+1)\cdot\beta]-[(1)]\sqcup[\alpha\cdot(a,b)\cdot\beta] \]
for any $a$, $b \geq 1$, though we will mostly apply it in the case $a=b=2$ indicated by the lemma.
\end{remark}
\begin{lemma}\label{lem:1tothemiddle}
	Let $c\geq 5$ and consider the caterpillar $[(2,1,2^{c-3},3,2^{c-2})]$. We have
	\[ [(2,1,2^{c-3},3,2^{c-2})]\doteq T_c-S_c+F_c+\sum_{i=1}^{c-4}([(2^i,3)]\sqcup[(2^{c-3-i},3,2^{c-2})]-[(2^{i+1})]\sqcup[(3,2^{c-4-i},3,2^{c-2})]).  \]
\end{lemma}
\begin{proof}
Applying \Cref{lem:shuffle1} to the $(2,1,2)$ pattern in the initial caterpillar ($\alpha$ is the empty composition), we have 
\begin{align*}
	[(2,1,2^{c-3},3,2^{c-2})]&\doteq [(2,2,1,2^{c-4},3,2^{c-2})]+ [(3)]\sqcup [(2^{c-3},3,2^{c-2})]-[(2,2)]\sqcup [(3,2^{c-5},3,2^{c-2})]\\
	&=[(2,2,1,2^{c-4},3,2^{c-2})]-[(2,2)]\sqcup [(3,2^{c-5},3,2^{c-2})]+F_c
\end{align*}
Reapplying \Cref{lem:shuffle1} to the first term on the right-hand side, we have
\begin{align*}
	[(2,1,2^{c-3},3,2^{c-2})]\doteq ([(2,2,2,1,2^{c-5},3,2^{c-2})] &+ [(2,3)]\sqcup [(2^{c-4},3,2^{c-2})]- [(2,2,2)]\sqcup [(1)\odot (2^{c-5},3,2^{c-2})]) \\
	&- [(2,2)]\sqcup [(3,2^{c-5},3,2^{c-2})] +F_c.
\end{align*}
If $c=5$ then this is exactly 
\[ [(2,1,2^{2},3,2^{3})]=T_5-S_5+F_5+\sum_{i=1}^1 [(2^i,3)]\sqcup[(2^{2-i},3,2^3)] - [(2^{i+1})]\sqcup[(3,3,2^3)].\]
Otherwise, continue to apply \Cref{lem:shuffle1} to the $(2,1,2)$ pattern in the first term on the right-hand side $c-5$ more times to obtain the claimed formula, once $T_c$ finally appears at the last application.
\end{proof}

\begin{lemma}\label{lem:X_Gc}
	For $c\geq 4$, the CSF $\mb{X}_{G_c}$ can be written as 
	\[{G_c}\doteq A_c-B_c+T_c-S_c+F_c+\sum_{i=1}^{c-4}([(2^i,3)]\sqcup[(2^{c-3-i},3,2^{c-2})]-[(2^{i+1})]\sqcup[(3,2^{c-4-i},3,2^{c-2})]).
	\]
\end{lemma}
\begin{proof}
The $c=4$ case is computed in \Cref{ex:G4}. For $c\geq 5$, as in the example, a first application of \eqref{eq:dottriple} to the incident edges $\set{f_2^L,v_1^R}$ and $\set{v_2^L,f_2^L}$ of $G_c$ yields 
\[G_c\doteq A_c-B_c+[(2,1,2^{c-3},3,2^{c-2})].\]
Then, applying \Cref{lem:1tothemiddle} to the last term on the right-hand side completes the proof.
\end{proof}

Now we will show that $\mb{X}_{H_c}$ can be expressed identically. As with $\mb{X}_{G_c}$, the $A_c$ and $B_c$ terms are obtained quickly, though the rest of the terms will now require more steps. We begin by introducing some auxiliary graphs that appear in the calculation.

\begin{definition}
	Let $R_c$ be the graph on $4c-4$ vertices obtained from $[(2^{c-1})]\sqcup[(2^{c-1})]$ by adding edges $\set{v_2^L,v_1^R}$ and $\set{v_{c-1}^L, f_1^R}$, and deleting edge $\set{v_1^R,f_1^R}$.
\end{definition}
For example, $R_5$ is depicted below.
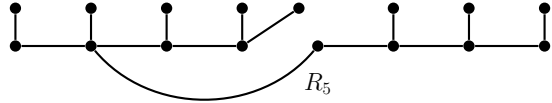
\begin{center}
	\begin{tikzpicture}[auto=center,every node/.style={circle, fill=black, scale=0.45}, style=thick, scale=0.50]
		\foreach \i [evaluate=\i as \j using int(2*\i)] in {1,...,8} {
			\node (S\i) at (\j,0) {};
		}
		\foreach \i [evaluate=\i as \j using int(2*\i)] in {1,2,3,6,7,8}{
			\node (L\i) at (\j,1) {};
			\draw (S\i) -- (L\i);
		}
		\foreach \i in {1,2,3,5,6,7} {
			\draw (S\i) -- (S\the\numexpr\i+1\relax);
		}
		\node (L4) at (8,1) {};
		\node (L5) at (9.5,1) {};
		\draw[thick] (S4) to (L4);
		\draw[thick] (S4) to (L5);
		\draw[thick, bend  right=50] (S2) to (S5);
		\node[fill=white] at (10,-1) {\Huge{$R_5$}};
	\end{tikzpicture}
\end{center}
\begin{lemma}\label{lem:Hc step 1}
	The CSF $\mb{X}_{H_c}$ can be expressed
	\[H_c \doteq A_c-B_c+R_c+[(1)]\sqcup[(2,1,2^{2c-4})]-[(1)]\sqcup[(2^{c-2},1,2^{c-1})].\]
\end{lemma}
\begin{proof}We compute using \eqref{eq:dottriple}, depicting $H_c$ as below.
		\begin{center}
		\begin{tikzpicture}[auto=center,every node/.style={circle, fill=black, scale=0.45}, style=thick, scale=0.50] 
			\foreach \i [evaluate=\i as \j using int(-9+2*\i)] in {1,2,4,5,6,8} {
				\node (S\i) at (\j,0) {};
			}
			\foreach \i [evaluate=\i as \j using int(-9+2*\i)] in {1,2,4,5,6,8} {
			\node (L\i) at (\j,1) {};
			\draw (S\i) -- (L\i);
		}			
			\draw  (S1) -- (-4,0);
			\draw[dotted] (-4,0) -- (-2,0);
			\draw (-2,0) -- (S4); 
			\draw (S5) -- (4,0);
			\draw (6,0) -- (S8);
			\draw (S4) -- (L5);
			\draw[dotted] (4,0) -- (6,0);
			\draw[bend right=50] (S2) to (S5);
			\node[fill=white] at (2,-1) {\Huge{$H_c$}};
		\end{tikzpicture}
	\end{center}
Hence,
\begin{align*}
\grapheq[0.25]{	\foreach \i [evaluate=\i as \j using int(-9+2*\i)] in {1,2,4,5,6,8} {
		\node[vx] (S\i) at (\j,0) {};
	}
	\foreach \i [evaluate=\i as \j using int(-9+2*\i)] in {1,2,4,5,6,8} {
		\node[vx] (L\i) at (\j,1) {};
		\draw (S\i) -- (L\i);
	}			
	\draw  (S1) -- (-4,0);
	\draw[dotted] (-4,0) -- (-2,0);
	\draw (-2,0) -- (S4); 
	\draw (S5) -- (4,0);
	\draw (6,0) -- (S8);
	\draw[red, very thick] (S4) -- (L5);
	\draw[YellowGreen, very thick] (S5) -- (L5);
	\draw[dotted] (4,0) -- (6,0);
	\draw[bend right=50] (S2) to (S5);
}&\doteq
\grapheq[0.25]{	\foreach \i [evaluate=\i as \j using int(-9+2*\i)] in {1,2,4,5,6,8} {
		\node[vx] (S\i) at (\j,0) {};
	}
	\foreach \i [evaluate=\i as \j using int(-9+2*\i)] in {1,2,4,5,6,8} {
		\node[vx] (L\i) at (\j,1) {};
		\draw (S\i) -- (L\i);
	}			
	\draw  (S1) -- (-4,0);
	\draw[dotted] (-4,0) -- (-2,0);
	\draw (-2,0) -- (S4); 
	\draw (S5) -- (4,0);
	\draw (6,0) -- (S8);
	\draw[cyan, very thick] (S4) -- (S5);
	\draw[YellowGreen, very thick] (S5) -- (L5);
	\draw[dotted] (4,0) -- (6,0);
	\draw[bend right=50] (S2) to (S5);
}+
\grapheq[0.25]{	\foreach \i [evaluate=\i as \j using int(-9+2*\i)] in {1,2,4,5,6,8} {
		\node[vx] (S\i) at (\j,0) {};
	}
	\foreach \i [evaluate=\i as \j using int(-9+2*\i)] in {1,2,4,6,8} {
		\node[vx] (L\i) at (\j,1) {};
		\draw (S\i) -- (L\i);
	}			
	\draw  (S1) -- (-4,0);
		\node[vx] (L5) at (1,1) {}; 
	\draw[dotted] (-4,0) -- (-2,0);
	\draw (-2,0) -- (S4); 
	\draw (S5) -- (4,0);
	\draw (6,0) -- (S8);
	\draw[red, very thick] (S4) -- (L5);
	\draw[dotted] (4,0) -- (6,0);
	\draw[bend right=50] (S2) to (S5);
}-
\grapheq[0.25]{	\foreach \i [evaluate=\i as \j using int(-9+2*\i)] in {1,2,4,5,6,8} {
		\node[vx] (S\i) at (\j,0) {};
	}
	\foreach \i [evaluate=\i as \j using int(-9+2*\i)] in {1,2,4,6,8} {
		\node[vx] (L\i) at (\j,1) {};
		\draw (S\i) -- (L\i);
	}			
	\node[vx] (L5) at (1,1) {}; 
	\draw  (S1) -- (-4,0);
	\draw[dotted] (-4,0) -- (-2,0);
	\draw (-2,0) -- (S4); 
	\draw (S5) -- (4,0);
	\draw (6,0) -- (S8);
	\draw[cyan, very thick] (S4) -- (S5);
	\draw[dotted] (4,0) -- (6,0);
	\draw[bend right=50] (S2) to (S5);
}\\ 
&\doteq A_c+R_c-\grapheq[0.3]{	\foreach \i [evaluate=\i as \j using int(-9+2*\i)] in {1,2,4,5,6,8} {
		\node[vx] (S\i) at (\j,0) {};
	}
	\foreach \i [evaluate=\i as \j using int(-9+2*\i)] in {1,2,4,6,8} {
		\node[vx] (L\i) at (\j,1) {};
		\draw (S\i) -- (L\i);
	}			
	\node[vx] (L5) at (1,1) {}; 
	\draw  (S1) -- (-4,0);
	\draw[dotted] (-4,0) -- (-2,0);
	\draw (-2,0) -- (S4); 
	\draw (S5) -- (4,0);
	\draw (6,0) -- (S8);
	\draw (S4) -- (S5);
	\draw[red, very thick] (L2) -- (S2);
	\draw[dotted] (4,0) -- (6,0);
	\draw[YellowGreen, bend right=50] (S2) to (S5);
}\\
&\doteq A_c+R_c-\grapheq[0.28]{	\foreach \i [evaluate=\i as \j using int(-9+2*\i)] in {1,2,4,5,6,8} {
		\node[vx] (S\i) at (\j,0) {};
	}
	\foreach \i [evaluate=\i as \j using int(-9+2*\i)] in {1,4,6,8} {
		\node[vx] (L\i) at (\j,1) {};
		\draw (S\i) -- (L\i);
	}			
	\node[vx] (L5) at (1,1) {}; 
	\node[vx] (L2) at (-5,1) {}; 
	\draw  (S1) -- (-4,0);
	\draw[dotted] (-4,0) -- (-2,0);
	\draw (-2,0) -- (S4); 
	\draw (S5) -- (4,0);
	\draw (6,0) -- (S8);
	\draw (S4) -- (S5);
	\draw[cyan, very thick] (L5) -- (S5);
	\draw[dotted] (4,0) -- (6,0);
	\draw[YellowGreen, bend right=50] (S2) to (S5);
}-(\grapheq[0.28]{	\foreach \i [evaluate=\i as \j using int(-9+2*\i)] in {1,2,4,5,6,8} {
	\node[vx] (S\i) at (\j,0) {};
}
\foreach \i [evaluate=\i as \j using int(-9+2*\i)] in {1,4,6,8} {
	\node[vx] (L\i) at (\j,1) {};
	\draw (S\i) -- (L\i);
}			
\node[vx] (L5) at (1,1) {}; 
\node[vx] (L2) at (-5,1) {}; 
\draw  (S1) -- (-4,0);
\draw[dotted] (-4,0) -- (-2,0);
\draw (-2,0) -- (S4); 
\draw (S5) -- (4,0);
\draw (6,0) -- (S8);
\draw (S4) -- (S5);
\draw[red, very thick] (L2) -- (S2);
\draw[dotted] (4,0) -- (6,0);
})
\\
& \hspace{2cm}+
\grapheq[0.28]{	\foreach \i [evaluate=\i as \j using int(-9+2*\i)] in {1,2,4,5,6,8} {
		\node[vx] (S\i) at (\j,0) {};
	}
	\foreach \i [evaluate=\i as \j using int(-9+2*\i)] in {1,4,6,8} {
		\node[vx] (L\i) at (\j,1) {};
		\draw (S\i) -- (L\i);
	}			
	\node[vx] (L5) at (1,1) {}; 
	\node[vx] (L2) at (-5,1) {}; 
	\draw  (S1) -- (-4,0);
	\draw[dotted] (-4,0) -- (-2,0);
	\draw (-2,0) -- (S4); 
	\draw (S5) -- (4,0);
	\draw (6,0) -- (S8);
	\draw (S4) -- (S5);
	\draw[cyan, very thick] (L5) -- (S5);
	\draw[dotted] (4,0) -- (6,0);
}\\
&\doteq A_c+R_c-B_c-[(1)]\sqcup[(2^{c-1},1,2^{c-2})]+[(1)]\sqcup[(2,1,2^{2c-4})],
\end{align*}	
as claimed.
\end{proof}
\begin{lemma}\label{lem:dot terms}For $c\geq 4$, 
	\[[(1)]\sqcup[(2,1,2^{2c-4})]-[(1)]\sqcup[(2^{c-2},1,2^{c-1})]\doteq \sum_{i=1}^{c-3}[(1)]\sqcup [(2^{i-1},3)]\sqcup [(2^{2c-3-i})] - [(1)]\sqcup [(2^{i+1})]\sqcup [(3,2^{2c-5-i})].
	\]
\end{lemma}
\begin{proof}
	Depict $[(1)]\sqcup[(2,1,2^{2c-4})]$ as below.
		\begin{center}
		\begin{tikzpicture}[auto=center,every node/.style={circle, fill=black, scale=0.45}, style=thick, scale=0.45] 
			\foreach \i [evaluate=\i as \j using int(-7+2*\i)] in {1,2,3,4,6} {
				\node (S\i) at (\j,0) {};
				\node (L\i) at (\j,1) {};
			}
			\foreach \i [evaluate=\i as \j using int(-7+2*\i)] in {1,3,4,6} {
				\draw (S\i) -- (L\i);
			}			
			\draw  (S1) -- (2,0);
			\draw[dotted] (2,0) -- (4,0);
			\draw (4,0) -- (S6); 
		\end{tikzpicture}
	\end{center}
 Then we compute, using \eqref{eq:dottriple} or \Cref{lem:shuffle1},
 \begin{align*}
	\grapheq[0.3]{\foreach \i [evaluate=\i as \j using int(-7+2*\i)] in {1,2,3,4,6} {
			\node[vx] (S\i) at (\j,0) {};
			\node[vx] (L\i) at (\j,1) {};
		}
		\foreach \i [evaluate=\i as \j using int(-7+2*\i)] in {1,3,4,6} {
			\draw (S\i) -- (L\i);
		}			
		\draw  (S1) -- (2,0);
		\draw[dotted] (2,0) -- (4,0);
		\draw (4,0) -- (S6);
		\draw[very thick, YellowGreen] (S2) -- (S3);
		\draw [very thick, red] (S3) -- (L3); } &\doteq \grapheq[0.3]{\foreach \i [evaluate=\i as \j using int(-7+2*\i)] in {1,2,3,4,6} {
			\node[vx] (S\i) at (\j,0) {};
			\node[vx] (L\i) at (\j,1) {};
		}
		\foreach \i [evaluate=\i as \j using int(-7+2*\i)] in {1,4,6} {
			\draw (S\i) -- (L\i);
		}			
		\draw  (S1) -- (2,0);
		\draw[dotted] (2,0) -- (4,0);
		\draw (4,0) -- (S6);
		\draw[very thick, YellowGreen] (S2) -- (S3);
		\draw [very thick, cyan] (S2) -- (L2); }
		+\grapheq[0.3]{\foreach \i [evaluate=\i as \j using int(-7+2*\i)] in {1,2,3,4,6} {
				\node[vx] (S\i) at (\j,0) {};
				\node[vx] (L\i) at (\j,1) {};
			}
			\foreach \i [evaluate=\i as \j using int(-7+2*\i)] in {1,4,6} {
				\draw (S\i) -- (L\i);
			}			
			\draw  (S1) -- (S2);
			\draw[dotted] (2,0) -- (4,0);
			\draw (S3) --(2,0);
			\draw (4,0) -- (S6);
			\draw[very thick, red] (L3) -- (S3);
			}
			- 
			\grapheq[0.3]{\foreach \i [evaluate=\i as \j using int(-7+2*\i)] in {1,2,3,4,6} {
					\node[vx] (S\i) at (\j,0) {};
					\node[vx] (L\i) at (\j,1) {};
				}
				\foreach \i [evaluate=\i as \j using int(-7+2*\i)] in {1,4,6} {
					\draw (S\i) -- (L\i);
				}			
				\draw  (S1) -- (S2);
				\draw (S3) -- (2,0);
				\draw[dotted] (2,0) -- (4,0);
				\draw (4,0) -- (S6);
				\draw [very thick, cyan] (S2) -- (L2); },\\
				[(1)]\sqcup[(2,1,2^{2c-4})]&\doteq [(1)]\sqcup [(2,2,1,2^{2c-5})]+[(1)]\sqcup[(3)]\sqcup[(2^{2c-4})]-[(1)]\sqcup[(2,2)]\sqcup[(3,2^{2c-6})].
	 \end{align*}
	 Continue to apply \Cref{lem:shuffle1} to the $(2,1,2)$ pattern of the first term on the right-hand side until the 1 is at position $c-1$. This means applying the modular relation of \Cref{lem:shuffle1} $c-3$ times to yield
	 \[[(1)]\sqcup[(2,1,2^{2c-4})]\doteq [(1)]\sqcup [(2^{c-2},1,2^{c-1})]+\sum_{i=1}^{c-3}[(1)]\sqcup[(2^{i-1},3)]\sqcup[(2^{2c-3-i})]-[(1)]\sqcup[(2^{i+1})]\sqcup[(3,2^{2c-5-i})],\]
	 as claimed.
\end{proof}
We continue with a lemma which will help us handle $\mb{X}_{R_c}$. It is similar to \Cref{lem:shuffle1} in that it lets us move leaf edges along the ``spine" of a graph that is  nearly a caterpillar. However, we use it to rewrite $\mb{X}_{R_c}$ in terms of the CSFs of only forests of caterpillars. 

To prepare, first we introduce some graph representation symbology. First, for any number $m$, let $\bar{m}=m-1$. If we wish to show that a spine vertex $v$ has $\bar{m}$ leaves attached, we may represent it as below. 
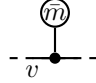
\begin{center}
	\begin{tikzpicture}[auto=center, style=thick, scale=0.6] 
	\node[ style={circle, fill=black, scale=0.45}] (A1) at (0,0) {};
	\node[style = {circle}, minimum size=6pt, inner sep=0pt, draw=black] (A2) at (0,1) {$\bar{m}$};
	\draw(A1) -- (A2);
	\draw[dashed] (-1,0) -- (1,0);
	\draw (-0.5,0) -- (0.5,0);
	\node at (-0.5,-0.3) {$v$};
\end{tikzpicture}
\end{center}
For instance, the caterpillar $[(4,3,2,1,2)]$ could be represented as below.
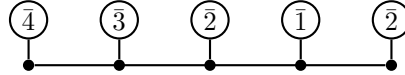
\begin{center}
	\begin{tikzpicture}[auto=center, style=thick, scale=0.6] 
		\node[ style={circle, fill=black, scale=0.45}] (A1) at (0,0) {};
		\node[ style={circle, fill=black, scale=0.45}] (A2) at (2,0) {};
		\node[ style={circle, fill=black, scale=0.45}] (A3) at (4,0) {};
		\node[ style={circle, fill=black, scale=0.45}] (A4) at (6,0) {};
		\node[ style={circle, fill=black, scale=0.45}] (A5) at (8,0) {};
		\node[style = {circle}, minimum size=14pt, inner sep=0pt, draw=black] (L1) at (0,1) {$\bar{4}$};
		\node[style = {circle}, minimum size=14pt, inner sep=0pt, draw=black] (L2) at (2,1) {$\bar{3}$};
		\node[style = {circle}, minimum size=14pt, inner sep=0pt, draw=black] (L3) at (4,1) {$\bar{2}$};
		\node[style = {circle}, minimum size=14pt, inner sep=0pt, draw=black] (L4) at (6,1) {$\bar{1}$};
		\node[style = {circle}, minimum size=14pt, inner sep=0pt, draw=black] (L5) at (8,1) {$\bar{2}$};
				
		\draw(A1) -- (A5);
			\foreach \i in {1,...,5} {
			\draw (A\i) -- (L\i);
		}			
	\end{tikzpicture}
\end{center}
To represent that a vertex $u$ is a neighbor to the extremal spine vertex $v_1$ of a caterpillar subgraph $[\alpha]=[(\alpha_1,\dots, \alpha_\ell)]$ (that is, $u$ is a neighbor of the spine vertex of the leaf component of size $\alpha_1$ but is not itself contained in the subgraph $[\alpha]$), or that vertex $v$ is a neighbor of $v_\ell$ (but not contained in $[\alpha]$), we will illustrate this as below.
\begin{center}
		\begin{tikzpicture}[auto=center, style=thick, scale=0.6] 
		\node[ style={circle, fill=black, scale=0.45}] (A1) at (-1,0) {};
		\node[style = {rectangle}, minimum size=10pt, inner sep=0pt, draw=black] (A2) at (0,0) {$\alpha$};
		\draw(A1) -- (A2);
		\node at (-1.3,-0.3) {$u$};
		\draw[dotted] (A1) -- (-2,0);
		\draw (A1) -- (-1.5,0);
	\end{tikzpicture}
	\hspace{3cm}
	\begin{tikzpicture}[auto=center, style=thick, scale=0.6] 
		\node[ style={circle, fill=black, scale=0.45}] (A1) at (0,0) {};
		\node[style = {rectangle}, minimum size=10pt, inner sep=0pt, draw=black] (A2) at (-1,0) {$\alpha$};
		\draw(A1) -- (A2);
		\node at (0.5,-0.3) {$v$};
		\draw[dotted] (A1) -- (1,0);
		\draw (A1) -- (0.5,0);
	\end{tikzpicture}
\end{center}

\begin{lemma} For any $a_1,a_2 \geq 1$ and any compositions $\alpha,\beta,\gamma$,
\[\grapheq[0.29]{
	\foreach \i [evaluate=\i as \j using int(-4+2*\i)] in {1, 2} {
	\node[vx] (S\i) at (\j,0) {};
	\node[lc] (L\i) at (\j,2) {$\bar{a}_{\i}$};
	\draw (S\i) to (L\i);
}
\node[cat] (S0) at (-4,0) {$\alpha$};
\node[cat] (S4) at (2,0) {$\beta$};
\node[vx]  (S5) at (4,0) {};
\node[cat] (S6) at (6,0) {$\gamma$};
\draw (S0) -- (S4);
\draw (S5) -- (S6);
\draw[very thick, YellowGreen] (S1) -- (S2);
\draw[bend right =50, red] (S1) to (S5);
}
\doteq 
\grapheq[0.29]{
	\foreach \i [evaluate=\i as \j using int(-4+2*\i)] in {1, 2} {
		\node[vx] (S\i) at (\j,0) {};
		\node[lc] (L\i) at (\j,2) {$\bar{a}_{\i}$};
		\draw (S\i) to (L\i);
	}
	\node[cat] (S0) at (-4,0) {$\alpha$};
	\node[cat] (S4) at (2,0) {$\beta$};
	\node[vx]  (S5) at (4,0) {};
	\node[cat] (S6) at (6,0) {$\gamma$};
	\draw (S0) -- (S4);
	\draw (S5) -- (S6);
	\draw[very thick, YellowGreen] (S1) -- (S2);
	\draw[bend right =50, cyan] (S2) to (S5);
}
\ +\
\grapheq[0.29]{
	\foreach \i [evaluate=\i as \j using int(-4+2*\i)] in {1, 2} {
		\node[vx] (S\i) at (\j,0) {};
		\node[lc] (L\i) at (\j,2) {$\bar{a}_{\i}$};
		\draw (S\i) to (L\i);
	}
	\node[cat] (S0) at (-4,0) {$\alpha$};
	\node[cat] (S4) at (2,0) {$\beta$};
	\node[vx]  (S5) at (4,0) {};
	\node[cat] (S6) at (6,0) {$\gamma$};
	\draw (S0) -- (S1);
	\draw (S2) -- (S4);
	\draw (S5) -- (S6);
	\draw[bend right =50,red] (S1) to (S5);
}\ -\ 
\grapheq[0.29]{
	\node[vx] (S1) at (-2,0) {};
	\node[lc] (L1) at (-2,2) {$\bar{a}_1$};
	\node[cat] (S0) at (-4,0) {$\alpha$};
	\node[cat] (S2) at (0,0) {$\beta^*$};
	\node[vx] (S4) at (2,0) {};
	\node[lc] (L4) at (2,2) {$\bar{a}_2$};
	\node[vx]  (S5) at (4,0) {};
	\node[cat] (S6) at (6,0) {$\gamma$};
	\draw (S0) -- (S1) -- (L1);
	\draw (S2) -- (S6);
	\draw (S4) -- (L4);
	\draw[very thick, cyan] (S4) -- (S5);
}.
\]
\label{lem:shuffle1b}
\end{lemma}
\begin{proof}
	The proof is by application of \eqref{eq:triplerelation} as indicated by the colored edges of the equation. Note the reversal of $\beta$ in the last term.
\end{proof}

\begin{lemma}\label{lem:Rc} For all $c\geq 4$,
	\[R_c\doteq T_c + \sum_{i=1}^{c-3}[(2^{c-3-i},3)]\sqcup [(2^{i+1},1,2^{c-2})]-[(2^{i+1})]\sqcup [(3,2^{c-3-i},1,2^{c-2})]. \]
\end{lemma}
\begin{proof}
	First apply \Cref{lem:shuffle1b} to edges $\set{v_2^L, v_1^R}$ and $\set{v_2^L,v_3^L}$ of $R_c$. Then apply the same lemma repeatedly to the unique term on the right-hand side which is not a forest of caterpillars until we finally reach the graph of the form below (in the notation of the lemma).
	\[		\grapheq[0.5]{\foreach \i [evaluate=\i as \j using int(-4+2*\i)] in {1, 2} {
				\node[vx] (S\i) at (\j,0) {};
				\node[lc] (L\i) at (\j,2) {$\bar{a}_{\i}$};
				\draw (S\i) to (L\i);
			}
			\node[cat] (S0) at (-4,0) {$\alpha$};
			\node[vx]  (S5) at (2,0) {};
			\node[cat] (S6) at (4,0) {$\gamma$};
			\draw (S0) -- (S2);
			\draw (S5) -- (S6);
			\draw[bend right =50] (S2) to (S5);
			}.\]
	This graph will have $\alpha=(2^{c-3})$, $a_1=2$, $a_2=3$ and $\gamma=(2^{c-2})$ so it is $T_c=[(2^{c-2},1,3,2^{c-2})]$.
	The rest of the terms are collected in the summation, as is straightforward to check.
\end{proof}
The next lemma allow us to rewrite the summation on the right-hand side in \Cref{lem:Rc} when combined with the result of \Cref{lem:dot terms}.
\begin{lemma} \label{lem:Hc step 2} For all $c\geq 4$, 
\begin{align*}\left(\sum_{i=1}^{c-3}[(2^{c-3-i},3)]\sqcup[(2^{i+1},1,2^{c-2})]-[(2^{i+1})]\sqcup [(3,2^{c-3-i},1,2^{c-2})]\right)+ [(1)]\sqcup[(2,1,2^{2c-4})]-[(1)]\sqcup[(2^{c-2},1,2^{c-1})]\\
	\doteq -S_c+F_c+\sum_{i=1}^{c-4}([(2^i,3)]\sqcup[(2^{c-3-i},3,2^{c-2})]-[(2^{i+1})]\sqcup[(3,2^{c-4-i},3,2^{c-2})]).
\end{align*}
\end{lemma}
\begin{proof}
 Apply \Cref{lem:tripleDNC} to the $(2,1,2)$ patterns in both terms of the summation on the left-hand side for all $1\leq i \leq c-3$:
 \begin{align*}
[(2^{c-3-i},3)]&\sqcup[(2^{i+1},1,2^{c-2})]-{[(2^{i+1})]\sqcup [(3,2^{c-3-i},1,2^{c-2})]}\\ & \doteq\left([(2^{c-3-i},3)]\sqcup[(2^{i},3,2^{c-2})]-\textcolor{blue}{[(2^{i+1})]\sqcup [(3,2^{c-4-i},3,2^{c-2})]}
\right)\\
 &\quad + \left([(2^{c-3-i},3)]\sqcup[(2^{i+1})]\sqcup[(3,2^{c-3})]-[(2^{i+1})]\sqcup[(3,2^{c-3-i})]\sqcup[(3,2^{c-3})]\right)\\
 &\qquad - \left([(1)]\sqcup[(2^{c-3-i},3)]\sqcup[(2^{i+1},2^{c-2})]-[(1)]\sqcup[(2^{i+1})]\sqcup[(3,2^{c-3-i},2^{c-2})]\right),
 \end{align*}
 which is valid except for if $i=c-3$. In this case, the applying the version of \Cref{lem:tripleDNC} indicated in \Cref{rk:genDNC} to the $(3,1,2)$ pattern in $[(2^{c-2})]\sqcup [(3,1,2^{c-2})]$, the term in blue is instead $[(2^{c-2})]\sqcup [(4,2^{c-2})]=S_c$, while also
 $$(2^{c-3-i},3)]\sqcup[(2^{i+1},1,2^{c-2})]=[(3)]\sqcup[(2^{c-2},1,2^{c-2})]=F_c.$$
Noting also that the terms on the second line of the right-hand side cancel, we have
\begin{align}
\sum_{i=1}^{c-3} [(2^{c-3-i},3)]&\sqcup[(2^{i+1},1,2^{c-2})]-[(2^{i+1})]\sqcup [(3,2^{c-3-i},1,2^{c-2})]  \nonumber \\
&\doteq F_c-S_c+\sum_{i=1}^{c-4}[(2^{c-3-i},3)]\sqcup[(2^{i},3,2^{c-2})]-[(2^{i+1})]\sqcup [(3,2^{c-4-i},3,2^{c-2})] \nonumber \\
&\qquad \quad -\sum_{i=1}^{c-3} \left([(1)]\sqcup[(2^{c-3-i},3)]\sqcup[(2^{i+1},2^{c-2})]-[(1)]\sqcup[(2^{i+1})]\sqcup[(3,2^{2c-5-i})]\right) \label{eq:reindexer}
\end{align}
Re-indexing the first terms of the summation on line \eqref{eq:reindexer} and applying \Cref{lem:dot terms}, this is 
\begin{align}
	\sum_{i=1}^{c-3} [(2^{c-3-i},3)]&\sqcup[(2^{i+1},1,2^{c-2})]-[(2^{i+1})]\sqcup [(3,2^{c-3-i},1,2^{c-2})]\nonumber   \\
	&\doteq-S_c+F_c+ \sum_{i=1}^{c-4}[(2^{c-3-i},3)]\sqcup[(2^{i},3,2^{c-2})]-[(2^{i+1})]\sqcup [(3,2^{c-4-i},3,2^{c-2})] \label{eq:reindexer2}  \\
	&\qquad -\left( [(1)]\sqcup[(2,1,2^{2c-4})]-[(1)]\sqcup[(2^{c-2},1,2^{c-1})]\right).\nonumber
\end{align}
Thus, the left-hand side of the formula of the lemma statement (after reindexing the first terms of the summation on line \eqref{eq:reindexer2}) is equal to
\begin{align*}
	-S_c+F_c+\sum_{i=1}^{c-4} \left( (2^{i},3)\sqcup(2^{c-3-i},3,2^{c-2})-(2^{i+1})\sqcup (3,2^{c-4-i},3,2^{c-2})\right)
\end{align*}
as claimed.
\end{proof}
\begin{example}
	We illustrate the computation of the previous two lemmas for $R_4$. First, applying \Cref{lem:Hc step 1} (or a single application of \cref{eq:dottriple}), we have
	\begin{align*}
		\grapheq[0.3]{
			\foreach \i [evaluate=\i as \j using int(2*\i)] in {1,...,6} {
				\node[vx] (S\i) at (\j,0) {};
			}
			\foreach \i [evaluate=\i as \j using int(2*\i)] in {1,2,3,5,6}{
				\node[vx] (L\i) at (\j,1) {};
				\draw (S\i) -- (L\i);
			}
			\foreach \i in {1,2,4,5} {
				\draw (S\i) -- (S\the\numexpr\i+1\relax);
			}
			\node[vx] (L4) at (7,1) {};
			\draw[thick] (S3) to (L4);
			\draw[very thick, YellowGreen] (S2) to (S3);
			\draw[very thick, bend  right=50, red] (S2) to (S4);
		} &\doteq \grapheq[0.3]{
			\foreach \i [evaluate=\i as \j using int(2*\i)] in {1,...,6} {
				\node[vx] (S\i) at (\j,0) {};
			}
			\foreach \i [evaluate=\i as \j using int(2*\i)] in {1,2,3,5,6}{
				\node[vx] (L\i) at (\j,1) {};
				\draw (S\i) -- (L\i);
			}
			\foreach \i in {1,2,4,5} {
				\draw (S\i) -- (S\the\numexpr\i+1\relax);
			}
			\node[vx] (L4) at (7,1) {};
			\draw[thick] (S3) to (L4);
			\draw[very thick, YellowGreen] (S2) to (S3);
			\draw[very thick, cyan] (S3) to (S4);
		}\ + \ 
		\grapheq[0.3]{
			\foreach \i [evaluate=\i as \j using int(2*\i)] in {1,...,6} {
				\node[vx] (S\i) at (\j,0) {};
			}
			\foreach \i [evaluate=\i as \j using int(2*\i)] in {1,2,3,5,6}{
				\node[vx] (L\i) at (\j,1) {};
				\draw (S\i) -- (L\i);
			}
			\foreach \i in {1,4,5} {
				\draw (S\i) -- (S\the\numexpr\i+1\relax);
			}
			\node[vx] (L4) at (7,1) {};
			\draw[thick] (S3) to (L4);
			\draw[very thick, bend  right=50, red] (S2) to (S4);
		}\ - \
		\grapheq[0.3]{
			\foreach \i [evaluate=\i as \j using int(2*\i)] in {1,...,6} {
				\node[vx] (S\i) at (\j,0) {};
			}
			\foreach \i [evaluate=\i as \j using int(2*\i)] in {1,2,3,5,6}{
				\node[vx] (L\i) at (\j,1) {};
				\draw (S\i) -- (L\i);
			}
			\foreach \i in {1,4,5} {
				\draw (S\i) -- (S\the\numexpr\i+1\relax);
			}
			\node[vx] (L4) at (7,1) {};
			\draw[thick] (S3) to (L4);
			\draw[very thick, cyan] (S3) to (S4);
		},\\
		&\doteq T_4 + [(3)]\sqcup [(2^{2},1,2^{2})]-[(2^{2})]\sqcup [(3,1,2^{2})].
	\end{align*}
Then the calculations of \Cref{lem:Hc step 2} applied to the last two terms yield
\begin{align*}
	R_4&\doteq T_4+ [(3)]\sqcup[(2,3,2^2)]-[(2^2)]\sqcup[(4,2^2)]-[(1)]\sqcup[(3)]\sqcup[(2^2,2^2)]+[(1)]\sqcup[(2^2)]\sqcup[(3,2^2)]\\
	&\doteq T_4-S_4+F_4-[(1)]\sqcup[(2,1,2^4)]+[(1)]\sqcup [(2^2,1,2^3)],
\end{align*}
where the last line uses the identity of \Cref{lem:dot terms}
\end{example}
We can now prove our formula.
\begin{theorem} \label{thm:GcHc}
	\label{thm:sameCSF}
	For any $c\geq 4$, the CSF $\mb{X}_{H_c}$ can be written as 
	\[{H_c}\doteq A_c-B_c+T_c-S_c+F_c+\sum_{i=1}^{c-4}([(2^i,3)]\sqcup[(2^{c-3-i},3,2^{c-2})]-[(2^{i+1})]\sqcup[(3,2^{c-4-i},3,2^{c-2})]).\]
	Hence, $\mb{X}_{G_c}=\mb{X}_{H_c}$.
\end{theorem}
\begin{proof}

Substituting for $R_c$ using \Cref{lem:Rc} in the expression of \Cref{lem:Hc step 1}, we have

\begin{align*}
H_c \doteq A_c-B_c+T_c+&\left(\sum_{i=1}^{c-3}[(2^{c-3-i},3)]\sqcup[(2^{i+1},1,2^{c-2})]-[(2^{i+1})]\sqcup [(3,2^{c-3-i},1,2^{c-2})]\right)\\
&+[(1)]\sqcup[(2,1,2^{2c-4})]-[(1)]\sqcup[(2^{c-2},1,2^{c-1})].
\end{align*}
Then, \Cref{lem:Hc step 2} applies directly to terms on the right-hand side to give
\[H_c\doteq A_c-B_c+T_c-S_c+F_c+ \sum_{i=1}^{c-4}([(2^i,3)]\sqcup[(2^{c-3-i},3,2^{c-2})]-[(2^{i+1})]\sqcup[(3,2^{c-4-i},3,2^{c-2})]),\]
as claimed, matching the expression of $\mb{X}_{G_c}$ given in \Cref{lem:X_Gc}.
\end{proof}
\begin{corollary}
	There exist pairs of non-isomorphic unicyclic graphs of arbitrary girth with the same CSF. That is, the CSF does not distinguish $c$-unicyclic graphs for any $c\geq 3$.
\end{corollary}
\begin{corollary}
	There exist infinitely many pairs of non-isomorphic bipartite (unicyclic) graphs with the same CSF. In particular, the CSF does not distinguish bipartite graphs on any number of vertices which is divisible by 4.
\end{corollary}

	\printbibliography

\end{document}